\documentclass[reqno]{amsart}
\usepackage{amssymb}
\usepackage{diagbox}
\usepackage{booktabs}
\usepackage{bookmark}
\usepackage{eucal}
\usepackage{amsmath}
\usepackage{amscd}
\usepackage[dvips]{color}
\usepackage{multicol}
\usepackage[all]{xy}           %xypic macro for latex2.09
\usepackage{graphicx}
\usepackage{color}
\usepackage{colordvi}
\usepackage{xspace}
\usepackage{array}

\usepackage[active]{srcltx} %SRC Specials for DVI Searching

\usepackage{hyperref}
\usepackage{lipsum}

\newcommand{\nc}{\newcommand}
\newcommand{\delete}[1]{}
\nc{\mfootnote}[1]{\footnote{#1}} % Use this to show footnotes
\nc{\todo}[1]{\tred{To do:} #1}

\delete{
\nc{\mlabel}[1]{\label{#1}}  % Use this to suppress names
\nc{\mcite}[1]{\cite{#1}}  % Use this to suppress names
\nc{\mref}[1]{\ref{#1}}  % Use this to suppress names
\nc{\mbibitem}[1]{\bibitem{#1}} % Use this to show number
}

\nc{\mlabel}[1]{\label{#1}  % Use the next two lines to show names
{\hfill \hspace{1cm}{\bf{{\ }\hfill(#1)}}}}
\nc{\mcite}[1]{\cite{#1}{{\bf{{\ }(#1)}}}}  % Use this lines to show names
\nc{\mref}[1]{\ref{#1}{{\bf{{\ }(#1)}}}}  % Use this lines to show names
\nc{\mbibitem}[1]{\bibitem[\bf #1]{#1}} % Use this to show name
\newtheorem{theorem}{Theorem}[section]
\newtheorem{example}{Example}[section]
\newtheorem{definition}[theorem]{Definition}
\newtheorem{lemma}[theorem]{Lemma}

\newtheorem{prop-def}[theorem]{Proposition-Definition}
\newtheorem{assertion}{\ \ \ Assertion}[theorem]
\newtheorem{remark}[theorem]{Remark}

\newtheorem{proposition}[theorem]{Proposition}

\makeatletter
\@addtoreset{equation}{section}
\makeatother
\nc{\tred}[1]{\textcolor{red}{#1}}
\nc{\tblue}[1]{\textcolor{blue}{#1}}
\nc{\tgreen}[1]{\textcolor{green}{#1}}
\nc{\tpurple}[1]{\textcolor{purple}{#1}}
\nc{\btred}[1]{\textcolor{red}{\bf #1}}
\nc{\btblue}[1]{\textcolor{blue}{\bf #1}}
\nc{\btgreen}[1]{\textcolor{green}{\bf #1}}
\nc{\btpurple}[1]{\textcolor{purple}{\bf #1}}

\nc{\twovec}[2]{\left(\begin{array}{c} #1 \\ #2\end{array} \right )}
\nc{\threevec}[3]{\left(\begin{array}{c} #1 \\ #2 \\ #3 \end{array}\right )}
\nc{\twomatrix}[4]{\left(\begin{array}{cc} #1 & #2\\ #3 & #4 \end{array} \right)}
\nc{\threematrix}[9]{{\left(\begin{matrix} #1 & #2 & #3\\ #4 & #5 & #6 \\ #7 & #8 & #9 \end{matrix} \right)}}
\nc{\twodet}[4]{\left|\begin{array}{cc} #1 & #2\\ #3 & #4 \end{array} \right|}
\nc{\rk}{\mathrm{r}}

\nc{\redtext}[1]{\textcolor{red}{#1}}

\makeatletter
\g@addto@macro{\endabstract}{\@setabstract}
\newcommand{\authorfootnotes}{\renewcommand\thefootnote{\@fnsymbol\c@footnote}}%
\makeatother
\begin{document}
\begin{center}
  \LARGE
Post-Lie algebra structures on the Witt algebra  \par \bigskip

  \normalsize
  \authorfootnotes
Xiaomin Tang   \footnote{Corresponding author: {\it X. Tang. Email:} x.m.tang@163.com}

  Department of Mathematics, Heilongjiang University,
Harbin, 150080, P. R. China

\end{center}

%%%%%%%%%%%%%%%%%%%%%%

\begin{abstract}
In this paper, we characterize the graded post-Lie algebra structures and a class of
shifting post-Lie algebra structures on the Witt algebra.
We obtain some new Lie algebras and give a class of their modules.  As an application, the homogeneous Rota-Baxter operators and a class of
non-homogeneous Rota-Baxter operators of weight $1$ on the  Witt algebra are studied.
\vspace{2mm}

\noindent{\it Keywords:} Lie algebra, post-Lie algebra,   Witt algebra, Rota-Baxter operator
\vspace{2mm}

\noindent{\it MSC}: 17A30, 17B68, 17D25.

\end{abstract}
%\baselineskip=20pt

%\keywords{Rota-Baxter operators; Yang-Baxter equation;
%Left-symmetric algebra;  associative algebra }

%\subjclass[2000]{16W99, 05C05, 08A50}

%\bf {MSC}:

%\maketitle

%\tableofcontents

\setcounter{section}{0}
{\ }

 \baselineskip=20pt

%\tableofcontents

%%%%%%%%%%%%%%%%%%%%%%%%%%%%%%%%%%%%%%      text      %%%%%%%%%%%%%%%%%%%%%%%%%%%%%%%%%%%%%%%%%%
\section{Introduction and preliminaries}

In recent years, post-Lie algebras have aroused the interest of a great many authors, see \cite{BGN,Burde2,Burde3,Burde1,Eb,Eb1,Mun,pan,tang2014}.
Post-Lie algebras were introduced around 2007 by B. Vallette  \cite{Vall}, who found
the structure in a purely operadic manner as the Koszul dual of a commutative trialgebra.
Moreover, Vallette \cite{Vall} proved that post-Lie algebras have the important algebraic property
of being Koszul. As pointed out by Hans Z. Munthe-Kaas \cite{Mun},
post-Lie algebras also arise naturally from the differential
geometry of homogeneous spaces and Klein geometries, topics that are closely
related to Cartan¡¯s method of moving frames.  In addition,  post-Lie algebras also turned up in relations with Lie groups \cite{Burde1,Mun}, classical Yang-Baxter equation \cite{BGN}, Hopf algebra and classical r-matrices \cite{Eb2} and Rota-Baxter operators\cite{Galb}.

One of the most important problems in the study of post-Lie algebras is to find the post-Lie algebra structures on the (given) Lie algebras. In \cite{pan}, the authors determined all post-Lie algebra structures on $sl(2, \mathbb{C})$ of special linear Lie algebra of order $2$, and in \cite{tang2014}, by using the Gr\"{o}bner basis the package in computer algebra software Maple, they studied the post-Lie algebra structures on the solvable Lie algebra $t(2, \mathbb{C})$ (the Lie algebra of $2\times 2$ upper triangular matrices) and give all $65$ types of these post-Lie algebra structures.

It may be useful and interesting for characterizing the post-Lie algebra structures
on some important Lie algebras. As we see, most of the study on post-Lie algebras has been focused on the finite-dimensional
case. It is natural to consider the infinite-dimensional case.  In particular, the authors \cite{Galb} obtained a class of induced post-Lie algebra structures on the Witt algebra by use of the homogeneous Rota-Baxter operators.  As a matter of fact, the Rota-Baxter operators were originally defined on associative algebras by G. Baxter to solve an analytic formula in probability \cite{baxter} and then developed by the Rota school \cite{rota}. These operators have showed up in many areas in mathematics and mathematical physics (see \cite{chu,Galb,guol,tang20144} and the references therein).

Recall that the Witt algebra $W$ is an important infinite-dimensional Lie algebra with the $\mathbb{C}$-basis $\{L_m| m\in \mathbb{Z} \}$ and the Lie brackets are defined by
\begin{equation*}
[L_m,L_n]=(m-n)L_{m+n}.
\end{equation*}
The Witt algebra occurs in the study of conformal field theory and plays an important role in many areas of mathematics and physics. Based on this background, we will study the post-Lie algebra structures on $W$ in this paper.  Below we will denote by $\mathbb{C}$ and $\mathbb{Z}$ the complex number field and the set of integer numbers, respectively. For a fixed integer $k$,  let $\mathbb{Z}_{>k}=\{t\in \mathbb{Z}|t>k\}$,   $\mathbb{Z}_{<k}=\{t\in \mathbb{Z}|t<k\}$,  $\mathbb{Z}_{\geqslant k}=\{t\in \mathbb{Z}|t\geqslant k\}$ and $\mathbb{Z}_{\leqslant k}=\{t\in \mathbb{Z}|t\leqslant k\}$. We also assume that the field in this paper always is the complex number field  $\mathbb{C}$ since  the Witt algebra is defined over  $\mathbb{C}$.
We now turn to the definition of post-Lie algebra following reference \cite{Vall}.

\begin{definition}\label{defi[]}
A post-Lie algebra $(V, \circ, [ , ])$ is a vector space $V$ over a field $k$ equipped with two
$k$-bilinear operations $x\circ y$ and $[x, y]$, such that $(V, [, ])$ is a Lie algebra and
\begin{eqnarray}
&&[x, y] \circ z =  x \circ (y \circ z)-y \circ (x \circ z)-<x,y> \circ z, \label{post1}\\
&&x\circ [y, z] = [x\circ y, z] + [y, x \circ z] \label{post2}
\end{eqnarray}
for all $x, y, z \in V$, where $<x,y>=x\circ y-y\circ x$.  We also say that  $(V, \circ, [ , ])$ is a post-Lie algebra structure on the Lie algebra $(V, [, ])$.  If a post-Lie algebra $(V, \circ, [ , ])$ such that $x\circ y=y\circ x$ for all $x,y\in V$, then it is called a commutative post-Lie algebra.
\end{definition}

\begin{definition}\label{iso}
Suppose that $(L, [,])$ is a Lie algebra. Two post-Lie algebras $(L, [,], \circ_1)$ and $(L, [,], \circ_2)$ on the Lie algebra $L$ are said to be isomorphic if there is an automorphism $\tau$ of the Lie algebra $(L, [,])$ such that
$\tau (x\circ_1 y)=\tau(x)\circ_2 \tau(y)$ for all $x,y\in L$.
\end{definition}

It is not difficult to verify the following proposition.

\begin{proposition}\label{equ1}
Let  $(V, \circ, [, ])$ be a post-Lie algebra defined by Definition \ref{defi[]}. Then the following operation
\begin{equation}\label{laoz}
\{x, y\} \triangleq <x,y> + [x, y],
\end{equation}
induces an another Lie algebra structure on $V$, where $<x,y>=x\circ y-y\circ x$. Furthermore, if two post-Lie algebras $(V, \circ_1, [, ])$ and
$(V, \circ_2, [, ])$ on the same Lie algebra $(V, [, ])$ are isomorphic, then the two induced Lie algebras $(V, \{, \}_1)$ and $(V, \{, \}_2)$
are isomorphic.
\end{proposition}

\begin{remark} \label{remark1}
The left multiplications of the post-Lie algebra $(V,[,], \circ)$ are denoted by $\mathcal{L}(x)$, i.e., we have $\mathcal{L}(x)(y) = x\circ y $ for all $x,y\in V$. By (\ref{post2}), we see that all operators $\mathcal{L}(x)$ are Lie algebra derivations of Lie algebra $(V,[,])$.
\end{remark}

Our results can be briefly summarized as follows: In Section 2, we classify the graded post-Lie algebra structures on the Witt algebra $W$, and then we
obtain the induced graded Lie algebras. In Section 3, we classify a class of shifting post-Lie algebra structures on the Witt algebra $W$, and give the induced non-graded Lie algebras.  In Section 4, we first recall the other definition of post-Lie algebra and give a new definition, and then the modules over some Lie algebras are given. In Section 5, we give the induced Rota-Baxter operators of weight $1$ from the post-Lie algebras on $W$.

\section{The graded post-Lie algebra structure On the Witt algebra}\label{grad}

Recently \cite{tangw22} proved that any commutative post-Lie algebra structure on the Witt algebra
$W$ is trivial (namely, $x\circ y=0$ for all $x,y\in W$). We now will dedicate to the study for the non-commutative case.
Since the Witt algebra is graded, it is also natural to suppose first
that the  algebras should be graded. Hence, in this section, we mainly
consider the graded post-Lie algebra structure on the Witt algebra $W$. Namely, we assume that it satisfies
\begin{equation}\label{mainequ1}
L_m\circ L_n=\phi(m,n)L_{m+n},\;\;\forall m,n\in \mathbb{Z},
 \end{equation}
where $\phi$ is a complex-value function on $\mathbb{Z}\times \mathbb{Z}$.

\begin{lemma}\label{ffmm}
There exists a graded post-Lie algebra structure on the Witt algebra $(W, [,])$ satisfying (\ref{mainequ1}) if and only if
there is a complex-value function $f$ on $\mathbb{Z}$ such that
\begin{eqnarray}
&& \phi(m,n)=(m-n)f(m), \label{1121} \\
&& (m-n)(f(m+n)+f(m)f(m+n)+f(n)f(m+n)-f(m)f(n))=0, \ \forall m,n\in \mathbb{Z}.  \label{f(mn)}
\end{eqnarray}
\end{lemma}

\begin{proof} The ``if'' part is easy to check. Next we  prove the ``only if'' part.
 By Remark \ref{remark1}, $\mathcal{L}(x)$ is a derivation of $W$. It is well known that every derivation of Witt algebra is inner \cite{zhulin}. So we have
\begin{equation}\label{mainequ2}
L_m\circ L_n=\mathcal{L}(L_m)(L_n)={\rm {ad}} (\psi (L_m)) L_n=[\psi(L_m),L_n]
\end{equation}
for some linear map $\psi$ from $W$ into itself.
Denote by $\psi (L_m)=\sum_{i\in \mathbb{Z}} k_i^{(m)}L_i$, where $k_i^{(m)}\in \mathbb{C}$ for any $i\in \mathbb{Z}$.  Then we have by (\ref{mainequ2}) that
$
L_m\circ L_n=\sum_{i\in \mathbb{Z}} (i-n)k_i^{(m)}L_{n+i}.
$
This, together with (\ref{mainequ1}), yields that $(i-n)k_i^{(m)}=0$ for any $i\in \mathbb{Z}\setminus\{m\}$ and $\phi(m,n)=(m-n)k_m^{(m)}$. This means that
$k_i^{(m)}=0$ for any $i\neq m$. Let $f(m)=k_m^{(m)}$ be a complex-valued function on $\mathbb{Z}$, which implies (\ref{1121}). Next, by (\ref{post1}) with a simple computation we obtain (\ref{f(mn)}).
\end{proof}

Let $(\mathcal{P}(\phi_i), \circ_i), i=1,2$ be two algebras on the Witt algebra $W=\mathcal{P}(\phi_i)$ equipped with $k$-bilinear operations $x\circ_i y$ such that
$L_m\circ_i L_n=\phi_i(m,n)L_{m+n}$ for all $m,n\in \mathbb{Z}$, where  $\phi_i, i=1,2$ are two complex-valued functions on $\mathbb{Z}\times \mathbb{Z}$. Furthermore, let $\tau: \mathcal{P}(\phi_1)\rightarrow \mathcal{P}(\phi_2)$ be a map given by $\tau(L_m)=-L_{-m}$ for all $m\in \mathbb{Z}$.
Clearly, $\tau$ is a Lie automorphism of the Witt algebra $(W, [,])$. Furthermore, we have

\begin{proposition}\label{iso1}
Let $(\mathcal{P}(\phi_i), \circ_i), i=1,2$ be two algebras and $\tau: \mathcal{P}(\phi_1)\rightarrow \mathcal{P}(\phi_2)$ be a map defined as above. Suppose that $\mathcal{P}(\phi_1)$ is a post-Lie algebra. Then $\mathcal{P}(\phi_2)$ is a post-Lie algebra and $\tau$ is an isomorphism from $\mathcal{P}(\phi_1)$ to $\mathcal{P}(\phi_2)$ if and only if $\phi_2(m,n)=-\phi_1(-m,-n)$.
\end{proposition}

\begin{proof} Suppose that $\mathcal{P}(\phi_2)$ is a post-Lie algebra and $\tau$ is an isomorphism from $\mathcal{P}(\phi_1)$ to $\mathcal{P}(\phi_2)$.
Then we have $\tau(L_m \circ_1 L_n)=-\phi_1(m,n)L_{-(m+n)}$ and $\tau(L_m)\circ_2 \tau(L_n)=L_{-m}\circ_2 L_{-n}=\phi_2(-m,-n)$. By $\tau(L_m \circ_1 L_n)=\tau(L_m)\circ_2 \tau(L_n)$, it follows that $\phi_2(m,n)=-\phi_1(-m,-n)$.

Conversely, suppose that $\phi_2(m,n)=-\phi_1(-m,-n)$ for all $m,n\in \mathbb{Z}$. Notice that $\mathcal{P}(\phi_1)$ is a post-Lie algebra, by Lemma \ref{ffmm} we know there is a complex-valued function $f_1$ on $\mathbb{Z}$ such that
\begin{eqnarray}
&& \phi_1(m,n)=(m-n)f_1(m), \label{tt1} \\
&& (m-n)(f_1(m+n)+f_1(m)f_1(m+n)+f_1(n)f_1(m+n)-f_1(m)f_1(n))=0,  \label{tt2}
\end{eqnarray}
for all $m,n\in \mathbb{Z}$.
By (\ref{tt1}), we have $\phi_2(m,n)=-\phi_1(-m,-n)=-(n-m)f_1(-m)$. Let $f_2(m)=f_1(-m)$ where $f_2$ is a complex-valued function on $\mathbb{Z}$, then it follows that
\begin{equation}\label{tt3}
 \phi_2(m,n)=(m-n)f_2(m).
\end{equation}
Furthermore, by (\ref{tt2}) and $f_2(m)=f_1(-m)$ we have that
\begin{equation}\label{tt4}
 (m-n)(f_2(m+n)+f_1(m)f_2(m+n)+f_2(n)f_2(m+n)-f_2(m)f_2(n))=0.
\end{equation}
Lemma \ref{ffmm} with (\ref{tt3}) and (\ref{tt4}) tells us that $\mathcal{P}(\phi_2)$ is a post-Lie algebra. The remainder is to prove that $\tau$ is an isomorphism. But one has
$$
\tau(L_m \circ_1 L_n)=-\phi_1(m,n)L_{-(m+n)}=\phi_2(-m,-n)L_{-(m+n)}=\tau(L_m) \circ_2 \tau(L_m),
$$
which completes the proof.
\end{proof}

For a complex-valued function $f$ on $\mathbb{Z}$, we denote $I,J$ by $$I=\{m\in \mathbb{Z}| f(m)=0\}, \ \ J=\{m\in \mathbb{Z}| f(m)=-1\}.$$

\begin{lemma}\label{IJ}
There exists a graded post-Lie algebra structure on the Witt algebra $(W, [,])$ satisfying (\ref{mainequ1}) if and only if
there is a complex-valued function $f$ on $\mathbb{Z}$ such that (\ref{1121}) holds and it satisfies

{\rm (i)} $m\in I\cup J$ for all $m\neq 0$, and

{\rm (ii)} $m,n\in I\Rightarrow m+n\in I$  and  $m+n\in J\Rightarrow m,n\in J$  for $m\neq n$.
\end{lemma}

\begin{proof}

The ``if'' part is easy to check. Next we  prove the ``only if'' part.  By Lemma \ref{ffmm}, there is a complex-valued function $f$ on $\mathbb{Z}$ satisfying (\ref{1121}) and (\ref{f(mn)}).
Let $n=0$ in (\ref{f(mn)}), we have $m(f(m)+f(m)^2)=0.$ Thus, for $m\neq 0$, one has $f(m)=0$ or $f(m)=-1$. This implies the conclusion of (i) holds. Now, we chose a pair of $m, n\in \mathbb{Z}$ with $m\neq n$, then it can be obtained by (\ref{f(mn)}) that
$$
f(m+n)+f(m)f(m+n)+f(n)f(m+n)-f(m)f(n)=0.
$$
It easy to see by the above equation that the conclusion of (ii) holds.
\end{proof}

Our main result of this section is the following.

\begin{theorem}\label{mmgg}
A graded post-Lie algebra structure satisfying (\ref{mainequ1}) on the Witt algebra $W$  must be one of the following types.

{\rm ($\mathcal{P}_1$)}  $L_m \circ_1 L_n=0$ for all $m,n\in \mathbb{Z}$;

{\rm ($\mathcal{P}_2$)}  $L_m \circ_2 L_n=(n-m)L_{m+n}$ for all $m,n\in \mathbb{Z}$;

{\rm ($\mathcal{P}_3^a$)}
$L_m \circ_3 L_n=\begin{cases}
(n-m)L_{m+n}, &  m> 0,\\
naL_n, & m=0, \\
0, & m<0;
\end{cases}
$

{\rm ($\mathcal{P}_4^a$)}
$L_m \circ_4 L_n=\begin{cases}
(n-m)L_{m+n}, &  m< 0,\\
naL_n, & m=0, \\
0, & m>0;
\end{cases}
$

{\rm ($\mathcal{P}_5$)}
$L_m \circ_5 L_n=\begin{cases}
(n-m)L_{m+n}, & m\geqslant 2,\\
0, & m\leqslant 1;
\end{cases}
$

{\rm ($\mathcal{P}_6$)}
$L_m \circ_6 L_n=\begin{cases}
(n-m)L_{m+n}, & m\geqslant 1,\\
0, & m\leqslant 2;
\end{cases}
$

{\rm ($\mathcal{P}_7$)}
$L_m \circ_7 L_n=\begin{cases}
(n-m)L_{m+n}, & m\leqslant -2,\\
0, & m\geqslant -1;
\end{cases}
$

{\rm ($\mathcal{P}_8$)}
$L_m \circ_8 L_n=\begin{cases}
(n-m)L_{m+n}, & m\geqslant -1,\\
0, & m\leqslant -2
\end{cases}
$ \\
where $a\in \mathbb{C}$.
Conversely, the above types are all the graded post-Lie algebra structure satisfying (\ref{mainequ1}) on the Witt algebra $W$.  Furthermore, the post-Lie algebras $\mathcal{P}_3^a$, $\mathcal{P}_5$ and $\mathcal{P}_6$ are isomorphic to the post-Lie algebras $\mathcal{P}_4^a$, $\mathcal{P}_7$ and $\mathcal{P}_8$ respectively, and other post-Lie algebras are not mutually isomorphic.
\end{theorem}

\begin{proof}
Suppose that $(W, [,], \circ)$ is a post-Lie algebra structure satisfying (\ref{mainequ1}) on the Witt algebra $W$. By Lemma \ref{IJ}, there is a complex-valued function $f$ on $\mathbb{Z}$ such that (\ref{1121}) holds and it satisfies (i) and (ii) in Lemma \ref{IJ}. We discuss the cases of $f(1),f(-1),f(2)$ and $f(-2)$. Lemma \ref{IJ} (i) tells us that $f(1),f(-1),f(2),f(-2)\in \{-1,0\}$, and so that  $2^4=16$ cases can happen. By a simple discussion, it can be seen that the $8$ cases listed in Tabular 1 are true. We here should point out that Theorem 2.22 of \cite{Galb} given another method to prove it since they have the same condition as (\ref{f(mn)}). Thus, by Lemma \ref{ffmm}, the graded post-Lie algebra structure on the Witt algebra $W$  must be one of the above
$8$ types.  Conversely, every type of the $8$ cases means that there is a complex-valued function $f$ on $\mathbb{Z}$ such that (\ref{1121}) holds and, the conclusions (i) and (ii) of Lemma \ref{IJ} are easily verified.  Thus, they are all the graded post-Lie algebra structure on the Witt algebra $W$ by Lemma \ref{IJ}.

\begin{table}[!hbp]
\begin{tabular}{c|cccc|c|c}
\toprule[1pt]
\diagbox{cases}{$f(n)$}& $f(1)$ & $f(-1)$ & $f(2)$ & $f(-2)$ & $f(0)$ & $f(n)$ for all $n$ \\
\toprule[1pt]
$\mathcal{P}_{1}$ &$0$ & $0$ & $0$ & $0$ & $0$ &$f(\mathbb{Z})=0$  \\
\hline
$\mathcal{P}_{2}$& $-1$ &  $-1$ &  $-1$ &  $-1$ &  $-1$ & $f(\mathbb{Z})=-1$  \\
\hline
$\mathcal{P}_{3}^a$& $-1$ & $0$ & $-1$ & $0$ & some $a\in \mathbb{C}$ & $f(\mathbb{Z}_{> 0})=-1$, $f(\mathbb{Z}_{< 0})=0$ and $f(0)=a$.  \\
\hline
$\mathcal{P}_{4}^a$&$0$ &  $-1$ & $0$ &  $-1$ & some $a\in \mathbb{C}$ & $f(\mathbb{Z}_{> 0})=0$, $f(\mathbb{Z}_{< 0})=-1$ and $f(0)=a$  \\
\hline
$\mathcal{P}_{5}$ &$0$ & $0$ & $-1$ & $0$ & $0$ &$f(\mathbb{Z}_{\geqslant 2})=-1$ and $f(\mathbb{Z}_{\leqslant 1})=0$  \\
\hline
$\mathcal{P}_{6}$& $-1$ & $-1$ & $0$ & $-1$ &$-1$  &  $f(\mathbb{Z}_{\geqslant 2})=0$ and $f(\mathbb{Z}_{\leqslant 1})=-1$  \\
\hline
$\mathcal{P}_{7}$ &$0$ & $0$ & $0$ & $-1$ & $0$ &$f(\mathbb{Z}_{\geqslant -1})=0$ and $f(\mathbb{Z}_{\leqslant -2})=-1$  \\
\hline
$\mathcal{P}_{8}$& $-1$ & $-1$ & $-1$ & $0$ & $-1$ & $f(\mathbb{Z}_{\geqslant -1})=-1$ and $f(\mathbb{Z}_{\leqslant -2})=0$  \\
 \bottomrule[1pt]
 \end{tabular}
\caption{ volume of $f(n)$}
\end{table}

Finally, by Proposition \ref{iso1} we know that the post-Lie algebras $\mathcal{P}_3^a$, $\mathcal{P}_5$ and $\mathcal{P}_6$ are isomorphic to the post-Lie algebras $\mathcal{P}_4^a$, $\mathcal{P}_7$ and $\mathcal{P}_8$ respectively.
We claim that $\mathcal{P}_3^{a_1}$ is not isomorphic to $\mathcal{P}_3^{a_2}$ when $a_1\neq a_2$. If not, then there a linear bijective map $\tau: \mathcal{P}_3^{a_1}\rightarrow \mathcal{P}_3^{a_2}$ as the isomorphism of post-Lie algebras. According to Definition \ref{iso}, $\tau$ first is an automorphism of the Lie algebra $(W,[,])$.  By the automorphisms of the classical Witt algebra \cite{Dok}, $\tau(L_m) = \epsilon c^mL_{\epsilon m}$ for all $m\in \mathbb{Z}$, where $c\in \mathbb{C}$ with $c\neq 0$ and $\epsilon\in \{\pm 1\}$. This, together with the definitions of $\mathcal{P}_3^{a}$, yields that $a_1=a_2$. This proves the claim above. Similarly, we have $\mathcal{P}_4^{a_1}$ is not isomorphic to $\mathcal{P}_4^{a_2}$ when $a_1\neq a_2$.
Clearly, the other post-Lie algebras are not mutually isomorphic.  The proof is completed.
\end{proof}

Due to Theorem \ref{mmgg} and Proposition \ref{equ1}, we now are able to give some Lie algebras on the space with $\mathbb{C}$-basis
$
\{
L_i|i\in \mathbb{Z}
\},
$
in which many cases are new and interesting.

\begin{proposition}\label{lie1}
Up to isomorphism, the post-Lie algebras in Theorem \ref{mmgg} give rise to the following Lie algebras on the space with $\mathbb{C}$-basis
$
\{
L_i|i\in \mathbb{Z}
\},
$ and with the bracket $\{, \}$ defined by
of Proposition \ref{equ1}:

{\rm ($\mathcal{LP}_1$):}  $\{L_m, L_n\}_1=(m-n)L_{m+n}$ for all $m,n\in \mathbb{Z}$;

{\rm ($\mathcal{LP}_3^a$):}
$
\{L_m , L_n\}_3=
\begin{cases}
(n-m)L_{m+n},  & m,n> 0, \\
(m-n)L_{m+n}, & m,n<0, \\
naL_n, & m=0, n> 0,\\
n(a-1)L_n &  m=0, n<0,\\
0,& \text{otherwise\ (unless some cases for}\ n=0 )
\end{cases}
$

{\rm ($\mathcal{LP}_5$):}
$
\{L_m , L_n\}_5=
\begin{cases}
(n-m)L_{m+n},  & m,n\geqslant 2, \\
(m-n)L_{m+n}, & m,n\leqslant 1, \\
0,& \text{otherwise}
\end{cases}
$\\
\ \ \ \ where $a \in \mathbb{ C}$.
\end{proposition}

\begin{proof}  Theorem \ref{mmgg} tells us that, up to isomorphism, there are $5$ types of graded post-Lie algebra structures satisfying (\ref{mainequ1}) on the Witt algebra, that is $\mathcal{P}_{1}$, $\mathcal{P}_{2}$, $\mathcal{P}^a_{3}$, $\mathcal{P}_{5}$ and $\mathcal{P}_{6}$.
By Proposition \ref{equ1} and a simple computation, we can obtain the $5$ types of Lie algebras denoted by  $\mathcal{LP}_{1}$, $\mathcal{LP}_{2}$, $\mathcal{LP}^a_{3}$, $\mathcal{LP}_{5}$ and $\mathcal{LP}_{6}$, respectively.  It is easy to verify that the Lie algebras $\mathcal{LP}_1$, $\mathcal{LP}_5$  are isomorphic to the Lie algebras $\mathcal{LP}_2$, $\mathcal{LP}_6$ respectively through the linear transformation $L_m\rightarrow -L_m$.  The conclusions are easily deducible.
\end{proof}

\begin{remark}
Proposition \ref{lie1} tells us that, up to isomorphism, there are $3$ types of Lie algebras induced by graded post-Lie algebras from Proposition \ref{equ1}. We note that this fact does not agree with Proposition 6.4 in \cite{Galb}. In fact, it is easy to check that every Lie algebra of Proposition 6.4 in \cite{Galb} has missed the same item $[L_m,L_n]$ in the compuation $\{L_m,L_n\}=L_m\circ L_n-L_n\circ L_m +[L_m,L_n]$. After correcting the small error, one can get the same results.
\end{remark}

\section{A class of shifting post-Lie algebra structures On the Witt algebra}\label{nograd}

In this section, we mainly
consider a class of non-graded post-Lie algebra structures on the Witt algebra $(W, [,])$. Namely, we assume that it satisfies
\begin{equation}\label{mainequ11}
L_m\circ L_n=\phi(m,n)L_{m+n}+\varrho(m,n)L_{m+n+\nu},\;\;\forall m,n\in \mathbb{Z},
 \end{equation}
where $\phi$ and $\varrho$ are complex-valued functions on $\mathbb{Z}\times \mathbb{Z}$ with $\varrho\neq 0$,  and $\nu$ is a fixed nonzero integer.
The motivation of studying this class of non-graded post-Lie algebra structures is inspired by \cite{tangb2012}, in which a class of non-graded
 left-symmetric algebraic structures on the Witt algebra has been considered. As in the theory of groups and in the theory of graded Lie algebras, we
 shall call such non-graded post-Lie algebra to be shifting post-Lie algebra since it with a shifting item. Our results show that the characterization of all non-graded post-Lie algebra structures on Witt algebra seems difficult.

\begin{lemma}\label{ffmm11}
There exists a shifting post-Lie algebra structure on the Witt algebra $(W, [,])$ satisfying (\ref{mainequ11}) if and only if
there are complex-valued functions $f$ and $g$ on $\mathbb{Z}$ such that
\begin{eqnarray}
&& \phi(m,n)=(m-n)f(m), \label{11211} \\
&& \varrho(m,n)=(m-n+\nu)g(m), \label{11222}\\
&& (m-n)(f(m+n)+f(m)f(m+n)+f(n)f(m+n)-f(m)f(n))=0, \label{11223}\\
&& (m-n)g(m)g(n)=((m-n+\nu)g(m)+(m-n-\nu)g(n))g(m+n+\nu),\label{jiaoj}\\
&& (m-n)(f(m)+f(n)+1)g(m+n)\nonumber\\
&=&(n-m+\nu)(f(m+n+\nu)-f(m))g(n)+(n-m-\nu)(f(m+n+\nu)-f(n))g(m). \label{f(mn1)}
\end{eqnarray}
\end{lemma}

\begin{proof}
We can finish the proof by a similar method to Lemma \ref{ffmm}.
\end{proof}

Let $(\mathcal{P}(\phi_i,\varrho_i,\nu_i), \circ_i), i=1,2$ be two algebras on the Witt algebra $W=\mathcal{P}(\phi_i,\varrho_i,\nu_i)$ equipped with $k$-bilinear operations $x\circ_i y$ such that
$L_m\circ_i L_n=\phi_i(m,n)L_{m+n}+\varrho_i(m,n)L_{m+n+\nu_i}$ for all $m,n\in \mathbb{Z}$, where  $\phi_i,\varrho_i, i=1,2$ are complex-valued functions on $\mathbb{Z}\times \mathbb{Z}$ and $\nu_i$ are nonzero integers. Furthermore, let $\tau: \mathcal{P}(\phi_1,\varrho_i,\nu_1)\rightarrow \mathcal{P}(\phi_2,\varrho_2,\nu_2)$ be a linear map given by $\tau(L_m)=-L_{-m}$ for all $m\in \mathbb{Z}$.

\begin{proposition}\label{iso2}
Let $\mathcal{P}(\phi_i,\varrho_i,\nu_i), i=1,2$ be two algebras and $\tau: \mathcal{P}(\phi_1,\varrho_i,\nu_1)\rightarrow \mathcal{P}(\phi_2,\varrho_2,\nu_2)$ be a map defined as above. Suppose that $\mathcal{P}(\phi_1,\varrho_1,\nu_1)$  is a post-Lie algebra. Then $\mathcal{P}(\phi_2, \varrho_2,\nu_2 )$ is a post-Lie algebra and $\tau$ is an isomorphism from $\mathcal{P}(\phi_1, \varrho_1,\nu_1)$ to $\mathcal{P}(\phi_2, \varrho_2,\nu_2 )$ if and only if
\begin{equation}\label{weiwei1}
\phi_2(m,n)=-\phi_1(-m,-n), \ \ \varrho_2 (m,n)=-\varrho_1(-m,-n), \ \ \nu_1=-\nu_2.
\end{equation}
\end{proposition}

\begin{proof}
The proof is similar to Proposition \ref{iso1}.
\end{proof}

Since a shifting post-Lie algebra structure on the Witt algebra $(W, [,])$ satisfying (\ref{mainequ11}) are entirely determined by
an integer number $\nu$ and two complex-valued functions $\phi$ and $\varrho$ on $\mathbb{Z}\times \mathbb{Z}$. According to (\ref{11211}) and (\ref{11222}) in Lemma \ref{ffmm11}, the classicization of such post-Lie algebras is dependent on the integer $\nu$ and two complex-valued functions $f$ and $g$ on $\mathbb{Z}$. It will be proved that the following Table 2 gives  all cases of $\nu, f$ and $g$, where $b$ is any nonzero complex number.

\begin{table}[!hbp]
\begin{tabular}{c|c|c|c|c|c}
\toprule[1pt]
& $f(\cdot)$& $f(\cdot)$ & $g(\cdot)$ & $g(\cdot)$ & $\nu$\\
\toprule[1pt]
$\mathcal{NP}^{b,\nu}_1$ &$f(\mathbb{Z}_{\geqslant 0})=-1$ & $f(\mathbb{Z}_{<0})=0$ & $g(-\nu)=-b$ & $g(\mathbb{Z}\setminus \{-\nu\})=0$ & $1, 2$ \\
\hline
$\mathcal{NP}^{b,\nu}_2$ &$f(\mathbb{Z}_{> 0})=-1$ & $f(\mathbb{Z}_{\leqslant 0})=0$ & $g(-\nu)=-b$ & $g(\mathbb{Z}\setminus \{-\nu\})=0$ & $-1, -2$ \\
\hline
$\mathcal{NP}^{b,\nu}_3$ &$f(\mathbb{Z}_{\geqslant 2})=-1$ & $f(\mathbb{Z}_{\leqslant 1})=0$ & $g(-\nu)=-b$ & $g(\mathbb{Z}\setminus \{-\nu\})=0$ & $-2, -3,-4$ \\
\hline
$\mathcal{NP}^{b,\nu}_4$ &$f(\mathbb{Z}_{\geqslant 2})=-1$ & $f(\mathbb{Z}_{\leqslant 1})=0$  & $g(2)=-b$ & $g(\mathbb{Z}\setminus \{2\})=0$ & $-1$ \\
\hline
$\mathcal{NP}^{b,\nu}_5$ &$f(\mathbb{Z}_{\geqslant 2})=-1$ & $f(\mathbb{Z}_{\leqslant 1})=0$ & $g(3)=2g(2)=-2b$ & $g(\mathbb{Z}\setminus \{2,3\})=0$ & $-2$ \\
\hline
$\mathcal{NP}^{b,\nu}_6$ &$f(\mathbb{Z}_{\geqslant 2})=0$ & $f(\mathbb{Z}_{\leqslant 1})=-1$ & $g(-\nu)=-b$ & $g(\mathbb{Z}\setminus \{-\nu\})=0$ & $-2, -3,-4$ \\
\hline
$\mathcal{NP}^{b,\nu}_7$ &$f(\mathbb{Z}_{\geqslant 2})=0$ & $f(\mathbb{Z}_{\leqslant 1})=-1$  & $g(2)=-b$ & $g(\mathbb{Z}\setminus \{2\})=0$ & $-1$ \\
\hline
$\mathcal{NP}^{b,\nu}_8$ &$f(\mathbb{Z}_{\geqslant 2})=0$ & $f(\mathbb{Z}_{\leqslant 1})=-1$  & $g(3)=2g(2)=-2b$ & $g(\mathbb{Z}\setminus \{2,3\})=0$ & $-2$ \\
\hline
$\mathcal{MP}^{b,\nu}_1$ &$f(\mathbb{Z}_{\leqslant 0})=-1$ & $f(\mathbb{Z}_{>0})=0$ & $g(-\nu)=-b$ & $g(\mathbb{Z}\setminus \{-\nu\})=0$ & $-1, -2$ \\
\hline
$\mathcal{MP}^{b,\nu}_2$ &$f(\mathbb{Z}_{< 0})=-1$ & $f(\mathbb{Z}_{\geqslant 0})=0$ & $g(-\nu)=-b$ & $g(\mathbb{Z}\setminus \{-\nu\})=0$ & $1, 2$ \\
\hline
$\mathcal{MP}^{b,\nu}_3$ &$f(\mathbb{Z}_{\leqslant -2})=-1$ & $f(\mathbb{Z}_{\geqslant -1})=0$ & $g(-\nu)=-b$ & $g(\mathbb{Z}\setminus \{-\nu\})=0$ & $2, 3,4$ \\
\hline
$\mathcal{MP}^{b,\nu}_4$ &$f(\mathbb{Z}_{\leqslant -2})=-1$ & $f(\mathbb{Z}_{\geqslant -1})=0$  & $g(-2)=-b$ & $g(\mathbb{Z}\setminus \{-2\})=0$ & $1$ \\
\hline
$\mathcal{MP}^{b,\nu}_5$ &$f(\mathbb{Z}_{\leqslant -2})=-1$ & $f(\mathbb{Z}_{\geqslant -1})=0$ & $g(-3)=2g(-2)=-2b$ & $g(\mathbb{Z}\setminus \{-2,-3\})=0$ & $2$ \\
\hline
$\mathcal{MP}^{b,\nu}_6$ &$f(\mathbb{Z}_{\leqslant -2})=0$ & $f(\mathbb{Z}_{\geqslant -1})=-1$ & $g(-\nu)=b$ & $g(\mathbb{Z}\setminus \{-\nu\})=0$ & $2, 3,4$ \\
\hline
$\mathcal{MP}^{b,\nu}_7$ &$f(\mathbb{Z}_{\leqslant -2})=0$ & $f(\mathbb{Z}_{\geqslant -1})=-1$  & $g(-2)=-b$ & $g(\mathbb{Z}\setminus \{-2\})=0$ & $1$ \\
\hline
$\mathcal{MP}^{b,\nu}_8$ &$f(\mathbb{Z}_{\leqslant -2})=0$ & $f(\mathbb{Z}_{\geqslant -1})=-1$  & $g(-3)=2g(-2)=-2b$ & $g(\mathbb{Z}\setminus \{-2,-3\})=0$ & $2$ \\
 \bottomrule[1pt]
 \end{tabular}
\caption{ volumes of $f(n), g(n)$ and $\nu$}
\end{table}

The above conclusion will be proved by some propositions as follows. First, notice that (\ref{11211}) and (\ref{jiaoj}) in Lemma \ref{ffmm11},  from Lemma \ref{ffmm} and Theorem \ref{mmgg} we have the following lemma.

\begin{lemma}
 Suppose that $(W, [,], \circ)$ is a shifting post-Lie algebra structure on the Witt algebra $(W, [,])$ satisfying (\ref{mainequ11}). Then $\phi$ satisfies (\ref{1121}) and $f$ must be determined by one of the cases $\mathcal{P}_1,\mathcal{P}_2$, $\mathcal{P}_3^a,\mathcal{P}_4^a$, $\mathcal{P}_4$-$\mathcal{P}_8$ in Tabel 1.
\end{lemma}

 Taking $n=-\nu$ in (\ref{jiaoj}), the following equation  is often used in our proof.
\begin{equation}\label{wupp3}
\nu g(m)g(-\nu)=(m+2\nu)g(m)g(m).
\end{equation}

\begin{proposition}\label{jin}
If $f$ takes the form determined by $\mathcal{P}_1$ or $\mathcal{P}_2$ in Tabel 1, then $g(\mathbb{Z})=0$. In this case, there is no any shifting
post-Lie algebra structure on the Witt algebra $(W, [,])$ satisfying (\ref{mainequ11}).
\end{proposition}

\begin{proof}
If $f$ takes the form determined by $\mathcal{P}_1$,  then $f(m)=0$ for all $m\in \mathbb{Z}$. Thus, by (\ref{f(mn1)}) we have $(m-n)g(m+n)=0$.
From this, we deduce that $g(\mathbb{Z})=0$. The case in which $f$ takes the form determined by $\mathcal{P}_2$ is similar.
\end{proof}

\begin{proposition}\label{bdd}
Suppose that $f$ takes the form determined by $\mathcal{P}_3^a$ in Table 1, i.e.,  $f(\mathbb{Z}_{>0})=-1, f(\mathbb{Z}_{<0})=0$ and $f(0)=a$ for some $a\in \mathbb{C}$. Then the shifting post-Lie algebra structure on the Witt algebra $(W, [,])$ satisfying (\ref{mainequ11}) is determined by $\mathcal{NP}^{b,\nu}_1$ or
$\mathcal{NP}^{b,\nu}_2$ in Table 2.
\end{proposition}

\begin{proof}
The proof is divided into the following: Assertions \ref{weib11}, \ref{wangy1}, \ref{pro21}.

\begin{assertion}\label{weib11}
{\rm (i)} When $\nu>0$, we have $g(\mathbb{Z}_{\geqslant 3})=0$ and $g(\mathbb{Z}_{\leqslant \min\{-3, -1-\nu\}})=0$;

{\rm (ii)} When $\nu<0$, we have $g(\mathbb{Z}_{\geqslant \max\{3, 1-\nu\}})=g(\mathbb{Z}_{\leqslant -3})=0$.
\end{assertion}

For all $m,n\in \mathbb{Z}$ such that $\{m,n, m+n+\nu\} \subset \mathbb{Z}_{>0}$ or
$\{m,n, m+n+\nu\} \subset \mathbb{Z}_{<0}$, by (\ref{f(mn1)}), it follows that $(m-n)g(m+n)=0$.
The results are easy to check.

\begin{assertion} \label{wangy1}
 If $\nu \geqslant 3 $ or $\nu \leqslant -3$, then $g(\mathbb{Z})=0$.
\end{assertion}

Case I. $\nu \geqslant 3$. By Assertion \ref{weib11} (i), $g(\mathbb{Z}_{\geqslant 3})=g(\mathbb{Z}_{\leqslant -1-\nu})=0$. Let $m=-n=1$ and $m=-n=2$ in (\ref{f(mn1)}), respectively, one can deduce $g(1)=g(2)=0$. Note that $-2\nu<-1-\nu$, one has $g(-2\nu)=0$. This, together with $m=-2\nu$ and $n=0$ in (\ref{jiaoj}), gives $0=-3\nu g(0)g(-\nu)$, so that $g(0)g(-\nu)=0$. From this, by letting $m=-\nu$ and $n=0$ in (\ref{jiaoj}), we obtain $-2\nu g(0)g(0)=0$. In other words, $g(0)=0$. In order to prove that $g(\mathbb{Z})=0$, it is enough to show that $g(m)=0$ for all $m\in \{-1,-2, \cdots, -\nu\}$.
Because of $-2\nu+1<-\nu-1$ and $-2\nu+2<-\nu-1$, one has $g(-2\nu+1)=g(-2\nu+2)=0$. This, together with $m=-2\nu+1, n=-1$ and $m=-2\nu+2, n=-2$ in (\ref{jiaoj}), respectively, gives
$$
(-3\nu+2)g(-1)g(-\nu)=(-3\nu+4)g(-2)g(-\nu)=0.
$$
As $\nu \geqslant 3$, we see that $-3\nu+2\neq 0$ and $-3\nu+4\neq 0$. This, together with the above equation, yields that $g(-1)g(-\nu)=g(-2)g(-\nu)=0$.
Next, by applying (\ref{wupp3}) to $m=-1$ and $m=-2$ respectively, we deduce $(-1+2\nu)g^2(-1)=(-2+2\nu)g^2(-2)=0$. Since $-1+2\nu\neq 0$ and $-2+2\nu\neq 0$, one has $g(-1)=g(-2)=0$.  If we let $m=-1$ and $n=-k$ in (\ref{f(mn1)}) where $k\geqslant2$, it follows that
\begin{equation}\label{wupp4}
(k-1)g(-1-k)=(1-k+\nu)f(-1-k+\nu)g(-k).
\end{equation}
By letting $k=2$ in (\ref{wupp4}), we have $g(-3)=0$. Again, by letting $k=3$ in (\ref{wupp4}), one has $g(-4)=0$. In turn, we obtain by (\ref{wupp4}) that
$g(\mathbb{Z}_{<0})=0$, therefore the conclusion is proved.

Case II. $\nu \leqslant -3$. Applying the similar method to Case I, one also can obtain that $g(\mathbb{Z})=0$.

\begin{assertion}\label{pro21}
$f$, $g$ and $\tau$ must be determined by $\mathcal{NP}^{b,\nu}_1$ or $\mathcal{NP}^{b,\nu}_2$ in Table 2.
\end{assertion}

By Assertion \ref{wangy1}, $g(\mathbb{Z})=0$ if $\nu\notin \{1,2,-1,-2\}$. Since $g\neq 0$, then $\nu\in \{1,2,-1,-2\}$.

Case 1. $\nu=1$. From Assertion \ref{weib11} (i), $g(\mathbb{Z}_{\geqslant 3})=g(\mathbb{Z}_{\leqslant -3})=0$. Next, we discuss the images $f(k)$
for $k\in \{-1,-2,0,1,2\}$. Taking $m=-n=1$ and $m=-n=2$ in (\ref{f(mn1)}), respectively, one has $g(1)=g(2)=0$. If we let $n=0$ in (\ref{jiaoj}), one has
\begin{equation}\label{guos1}
mg(m)g(0)=((m+1)g(m)+(m-1)g(0))g(m+1).
\end{equation}
Taking $m=-1, -2$ and $-3$ in (\ref{guos1}), respectively, by $g(-3)=0$  we have that
$$
g(-1)g(0)=2g(0)g(0), \ 2g(-2)g(0)=3g(-1)g(0), \ g(0)g(-2)=0,
$$
which implies that $g^2(0)=0$. Thus, $g(0)=0$.  If we let $m=-2, n=1$ and $m=-2, n=0$ in (\ref{f(mn1)}), respectively, then one has that
$(f(0)+1)g(-2)=0$ and $2(f(0)+1)g(-2)=f(0)g(-2)$. This yields that $g(-2)=0$. Let $m=-1$ and $n=0$ in (\ref{f(mn1)}), then
we have $(f(0)+1)g(-1)=0$. The fact of $g\neq 0$ means that $g(-1)=-b\neq 0$ for some
$b\in \mathbb{C}$. At the same time, we must have that $f(0)+1=0$, namely, $f(0)=a=-1$. By the above part of the analysis and discussion, we see that
$$f(\mathbb{Z}_{\geqslant 0})=-1, \ \ f(\mathbb{Z}_{<0})=0,\  g(-1)=-b, \ g(\mathbb{Z}\setminus \{-1\})=0,$$
for some nonzero $b\in \mathbb{C}$. This is of the type  $\mathcal{NP}^{b,1}_1$ in Table 2.

Case 2. $\nu=2$. By Assertion \ref{weib11} (i), $g(\mathbb{Z}_{\geqslant 3})=g(\mathbb{Z}_{\leqslant -3})=0$. Thus, we only discuss the images $f(k)$
for $k\in \{-1,-2,0,1,2\}$. Taking $m=-n=1$ and $m=-n=2$ in (\ref{f(mn1)}), respectively, it follows that $g(1)=g(2)=0$.  Let $m=-1, n=-2$ and $m=-3, n=-1$ in (\ref{jiaoj}), respectively, we get
$2g(-1)g(-2)=3g(-1)g(-1)$ and $g(-1)g(-2)=0$, which yields that $g^2(-1)=0$. Hence $g(-1)=0$.  We again obtain $g(0)=0$ by taking $m=-1$ and $n=0$ in (\ref{f(mn1)}). Finally, it follows that $(f(0)+1)g(-2)=0$ by taking $m=-2$ and $n=0$ in (\ref{f(mn1)}). Since $g\neq 0$, hence $f(0)=a=-1$ and $g(-2)=-b\neq 0$ for some nonzero $b\in \mathbb{C}$. It proves that
$$f(\mathbb{Z}_{\geqslant 0})=-1, \ \ f(\mathbb{Z}_{<0})=0,\  g(-2)=-b, \ g(\mathbb{Z}\setminus \{-2\})=0,$$
which is of the type  $\mathcal{NP}^{b,2}_1$ in Table 2.

Case 3. $\nu=-1$. Similar to the discussion of Case 1, we obtain the type $\mathcal{NP}^{b,-1}_2$ in Table 2.

Case 4. $\nu=-2$. Similar to the discussion of Case 2, we obtain the type $\mathcal{NP}^{b,-2}_2$ in Table 2.
\end{proof}

\begin{proposition}\label{zaa2}
Suppose that $f$ takes the form determined by $\mathcal{P}_5$ in Table 1, i.e., $f(\mathbb{Z}_{\geqslant2})=-1, f(\mathbb{Z}_{\leqslant 1})=0$.
Then the shifting post-Lie algebra structure on the Witt algebra $(W, [,])$ satisfying (\ref{mainequ11}) is determined by $\mathcal{NP}^{b,\nu}_3$,
$\mathcal{NP}^{b,\nu}_4$ or $\mathcal{NP}^{b,\nu}_5$ in Table 2.
\end{proposition}

\begin{proof} The proof is divided into the following: Assertions \ref{weib1}, \ref{wangy2} and \ref{pro2}.

\begin{assertion}\label{weib1}
{\rm (i)} When $\nu>0$, we have $g(\mathbb{Z}_{\geqslant 5})=g(\mathbb{Z}_{\leqslant 1-\nu})=0$;

{\rm (ii)} When $\nu<0$, we have $g(\mathbb{Z}_{ \geqslant \max\{5, 2-\nu})=g(\mathbb{Z}_{\leqslant 1})=0$.
\end{assertion}

For any $m,n\in \mathbb{Z}$ with $\{m,n, m+n+\nu\} \subset \mathbb{Z}_{\geqslant 2}$ or
$\{m,n, m+n+\nu\} \subset \mathbb{Z}_{\leqslant 1}$, one has by (\ref{f(mn1)}) that $(m-n)g(m+n)=0$.
The assertion are easy to check.

\begin{assertion}\label{wangy2}
If $\nu \geqslant 1 $ or $\nu \leqslant -5$,  then $g(\mathbb{Z})=0$.
\end{assertion}

Case I. $\nu \geqslant 1$. By Assertion \ref{weib1} (i), $g(\mathbb{Z}_{\geqslant 5})=g(\mathbb{Z}_{\leqslant 1-\nu})=0$. Taking $n=0$ in (\ref{f(mn1)}), then
\begin{equation}\label{wupp1}
(m(f(m)+1)-(m+\nu)f(m+\nu))g(m)=(\nu-m)(f(m+\nu)-f(m))g(0)
\end{equation}
for all $m\in \mathbb{Z}$. If we let $m\in\{2,3,4\}$, then $f(m)=f(m+\nu)=-1$ and by (\ref{wupp1}), we have $(m+\nu)g(m)=0$. This implies $g(2)=g(3)=g(4)=0$.
Take $m=1$ and $n=-1-\nu$ in (\ref{jiaoj}), then by $g(-1-\nu)=0$,
\begin{equation}\label{wupp2}
(2+2\nu)g(1)g(0)=0.
\end{equation}
If we let $m=1$ in (\ref{wupp1}), it follows that $-\nu g(1)=(1-\nu) g(0)$. This, together with (\ref{wupp2}), gives $g(1)=g(0)=0$ for the case $\nu>1$.
But for the case $\nu=1$, we also have $g(1)=0$ since $-\nu g(1)=(1-\nu) g(0)$ and $g(0)=0$ since $g(\mathbb{Z}_{\leqslant 1-\nu})=0$. In order to prove the conclusion, we only need show that $g(k)=0$ for all $k\in \{2-\nu, 3-\nu, \cdots, -2,-1\}$. For such $k$, we let $k=s-\nu$, where $2\leqslant s<\nu$.
Note that $k<0$, so $f(k)=0$ and $f(k+\nu)=f(s)=-1$. Applying (\ref{wupp1}) to $m=k$, we have by $g(0)=0$ that $(k-(k+\nu))g(k)=0$. In other words, $g(k)=0$, as desired.

Case II.  $\nu \leqslant -5$. By Assertion \ref{weib1} (ii), $g(\mathbb{Z}_{\geqslant 2-\nu})=g(\mathbb{Z}_{\leqslant 1})=0$. From this, we only need to prove that $g(m)=0$ for every $m\in \{2,3,\cdots, 1-\nu\}$.
For $m\in \mathbb{Z}$ with $1<m<2-\nu$, it follows that $m+2\nu<2+\nu<0$.
Taking $m=2$ and $n=-2-\nu$ in (\ref{jiaoj}) and (\ref{f(mn1)}), respectively, we obtain that
\begin{eqnarray}
&& (4+\nu)g(2)g(-2-\nu)=0, \label{guol1}\\
&& (4+\nu)g(-\nu)=(4+2\nu)g(2)+4g(-2-\nu)\label{guol2}.
\end{eqnarray}
Note that $4+\nu\neq 0$, so (\ref{guol1}) tells us that $g(2)g(-2-\nu)=0$. This, together with (\ref{guol2}), gives that
\begin{eqnarray}
&& (4+\nu)g(2)g(-\nu)=(4+2\nu)g^2(2)\label{guol3},\\
&& (4+\nu)g(-2-\nu)g(-\nu)=4g^2(-2-\nu)\label{guol4}.
\end{eqnarray}
If we let $m=2$ and $m=-2-\nu$ in (\ref{wupp3}), respectively, we have that
\begin{eqnarray}
&& \nu g(2)g(-\nu)=(2+2\nu)g^2(2)\label{guol5},\\
&& \nu g(-2-\nu)g(-\nu)=(\nu-2)g^2(-2-\nu)\label{guol6}.
\end{eqnarray}
It follows by (\ref{guol3}) and (\ref{guol5}) that
$$
\begin{pmatrix}
4+\nu & 4+2\nu\\
\nu & 2+2\nu
\end{pmatrix}
\begin{pmatrix}
g(2)g(-\nu)\\
-g^2(2)
\end{pmatrix}=0,
$$
and by (\ref{guol4}) and (\ref{guol6}) that
$$
\begin{pmatrix}
4+\nu & 4\\
\nu &\nu-2
\end{pmatrix}
\begin{pmatrix}
g(-2-\nu)g(-\nu)\\
-g^2(-2-\nu))
\end{pmatrix}=0.
$$
Since $\nu\leqslant -5$, we deduce
$$\det \begin{pmatrix}
4+\nu & 4+2\nu\\
\nu & 2+2\nu
\end{pmatrix}=8+6\nu\neq 0,
\ \
\det \begin{pmatrix}
4+\nu & 4\\
\nu &\nu-2
\end{pmatrix}=(\nu-4)(\nu+2)\neq 0,$$
which yields that $-g^2(2)=-g^2(-2-\nu)=0$, and then $g(2)=g(-2-\nu)=0$. Thus, from (\ref{guol2}) we have $(4+\nu)g(-\nu)=(4+2\nu)g(2)+4g(-2-\nu)=0$. Note that $4+\nu\neq 0$, so $g(-\nu)=0$. This, together with (\ref{wupp3}), implies that $(m+2\nu)g^2(m)=0$. For $m\in \mathbb{Z}$ with $1<m<2-\nu$, we have $m+2\nu<2+\nu<0$. Therefore, $g^2(m)=0$ and thereby $g(m)=0$ for every $m\in \{2,3,\cdots, 1-\nu\}$.
The proof is completed.

\begin{assertion}\label{pro2}
$f$, $g$ and $\tau$ must be determined by $\mathcal{NP}^{b,\nu}_3$, $\mathcal{NP}^{b,\nu}_4$ or $\mathcal{NP}^{b,\nu}_5$ in Table 2.
\end{assertion}

By Assertion \ref{wangy2}, $g(\mathbb{Z})=0$ if $\nu\notin \{-1, -2,-3,-4\}$. Since $g\neq0$, then $\nu\in \{-1, -2,-3,-4\}$.

Case 1. $\nu=-1$. From Clam \ref{weib1} (ii), $g(\mathbb{Z}_{\geqslant 5})=g(\mathbb{Z}_{\leqslant 1})=0$. Next, we discuss the images $f(k)$
for $k\in \{2,3,4\}$. Taking $m=0, n=3$ and $m=0, n=4$ in (\ref{f(mn1)}), respectively, one has that $g(3)=g(4)=0$. In this case,  $g(2)=-b\neq 0$ for some
$b\in \mathbb{C}$. Thus, it summarizes as
$$f(\mathbb{Z}_{\geqslant 2})=-1, \ \ f(\mathbb{Z}_{\leqslant 1})=0,\  g(2)=-b, \ g(\mathbb{Z}\setminus \{2\})=0.$$
It gives the type $\mathcal{NP}_4^{b,-1}$ in Table 2.

Case 2. $\nu=-2$. By Assertion \ref{weib1} (ii) we have $g(\mathbb{Z}_{\geqslant 5})=g(\mathbb{Z}_{\leqslant 1})=0$. Thus, we only discuss the images $f(k)$
for $k\in \{2,3,4\}$. Taking $m=0$ and $n=4$ in (\ref{f(mn1)}), it follows that $g(4)=0$.  If we let $m=3$ and $\nu=-2$ in (\ref{wupp3}), we get
$g(3)g(3)=2g(2)g(3)$.  This tells us that if $g(2)=0$ then $g(3)=0$. Since $g\neq 0$, it must be $g(2)=-b\neq 0$. In this case, either $g(3)=0$ which gives
the type $\mathcal{NP}_3^{b,-2}$ in Table 2; or $g(3)=2g(2)=-2b\neq 0$
which gives the type $\mathcal{NP}_5^{b,-2}$ in Table 2.

Case 3. $\nu=-3$. By Assertion \ref{weib1} (ii), $g(\mathbb{Z}_{\geqslant 5})=g(\mathbb{Z}_{\leqslant 1})=0$. Thus, we only discuss the images $f(k)$
for $k\in \{2,3,4\}$.  Taking $m=2, n=5$ and $m=2, n=4$ in (\ref{f(mn1)}), respectively, one has that
\begin{equation}\label{xiaomin1}
g(2)g(4)=0, \ \ g(3)g(4)=5g(2)g(3).
\end{equation}
If we let $\nu=-3$ and $m=2$ and $m=4$ in (\ref{wupp3}), respectively, we get
\begin{equation}\label{xiaomin2}
3g(2)g(3)=4g(2)g(2), \ \ 3g(2)g(3)=2g(4)g(4).
\end{equation}
Combing (\ref{xiaomin1}) with (\ref{xiaomin2}), it implies that $g^2(4)=10g^2(2)$ and $g(2)g(4)=0$. This means that $g(2)=g(4)=0$.
Since $g\neq 0$, we have $g(3)=-b\neq 0$ for some $b\in \mathbb{C}\setminus \{0\}$. It is easy to check that this is just the type $\mathcal{NP}_3^{b,-3}$ in Table 2.

Case 4. $\nu=-4$. By Assertion \ref{weib1} (ii) we have $g(\mathbb{Z}_{\geqslant 6})=g(\mathbb{Z}_{\leqslant 1})=0$. Thus, we only discuss the images $f(k)$
for $k\in \{2,3,4,5\}$.
If we let $\nu=-4$ and $m=2,3,5$ in (\ref{wupp3}), respectively, we get
\begin{equation}\label{xiaomin3}
4g(2)g(4)=2g(2)g(2), \ \ 4g(3)g(4)=5g(3)g(3), \ \ 4g(4)g(5)=3g(5)g(5).
\end{equation}
Taking $m=2$ and $n=6$ in (\ref{jiaoj}), it follows that $g(2)g(4)=0$. This, together with (\ref{xiaomin3}), implies that $g^2(2)=0$. So we have $g(2)=0$.
If we let $m=3, n=6$ and $m=3, n=5$ in (\ref{jiaoj}), respectively, we obtain that $g(3)g(5)=0$ and
$g(4)g(5)=3g(3)g(4)$. This yields that $g(4)g(5)=g(3)g(4)=0$. This, together with (\ref{xiaomin3}), gives that $g^2(3)=g^2(5)=0$. Thus,
$g(3)=g(5)=0$. Since $g\neq 0$, we have $g(4)=-b\neq 0$ for some $b\in \mathbb{C}\setminus \{0\}$. It is easy to check that this is just the type $\mathcal{NP}_3^{b,-4}$ in Table 2.
\end{proof}

\begin{proposition}\label{pro3}
Suppose that $f$ takes the form determined by $\mathcal{P}_6$ in Table 1.
Then the shifting post-Lie algebra structure on the Witt algebra $(W, [,])$ satisfying (\ref{mainequ11}) is determined by $\mathcal{NP}^{b,\nu}_6$,
$\mathcal{NP}^{b,\nu}_7$ or $\mathcal{NP}^{b,\nu}_8$ in Table 2.
\end{proposition}

\begin{proof}
The proof is similar to Proposition \ref{zaa2}.
\end{proof}

Next, by using Proposition \ref{iso2}, similar to Propositions \ref{bdd}, \ref{zaa2} and \ref{pro3}, one has the following three propositions.

\begin{proposition}\label{bdd1}
Suppose that $f$ takes the form determined by $\mathcal{P}_4^a$ in Table 1. Then the shifting post-Lie algebra structure on the Witt algebra $(W, [,])$ satisfying (\ref{mainequ11}) is determined by $\mathcal{MP}^{b,\nu}_1$ or
$\mathcal{MP}^{b,\nu}_2$ in Table 2.
\end{proposition}

\begin{proposition}\label{zaa22}
Suppose that $f$ takes the form determined by $\mathcal{P}_7$ in Table 1.
Then the shifting post-Lie algebra structure on the Witt algebra $(W, [,])$ satisfying (\ref{mainequ11}) is determined by $\mathcal{MP}^{b,\nu}_3$,
$\mathcal{MP}^{b,\nu}_4$ or $\mathcal{MP}^{b,\nu}_5$ in Table 2.
\end{proposition}

\begin{proposition}\label{pro32}
Suppose that $f$ takes the form determined by $\mathcal{P}_8$ in Table 1.
Then the shifting post-Lie algebra structure on the Witt algebra $(W, [,])$ satisfying (\ref{mainequ11}) is determined by $\mathcal{MP}^{b,\nu}_6$,
$\mathcal{MP}^{b,\nu}_7$ or $\mathcal{MP}^{b,\nu}_8$ in Table 2.
\end{proposition}

Our main result in this section is the following.

\begin{theorem}\label{postnon1}
A shifting post-Lie algebra structure satisfying (\ref{mainequ11}) on the Witt algebra $W$  must be one of the following types.

 {\rm ($\mathcal{NP}^{b,\nu}_1$):}  $\nu=1$ or $2$,

 \ \ \ \ \ \ $
L_m \circ_1 L_n=
\begin{cases}
(n-m)L_{m+n}, \ \ & m\geqslant 0,\\
nbL_{n}, & m=-\nu,
\\
0, & m<0,m\neq -\nu;
\end{cases}
$

{\rm ($\mathcal{NP}^{b,\nu}_2$):}  $\nu=-1$ or $-2$,

\ \ \ \ \ \  $
L_m \circ_2 L_n=
\begin{cases}
(n-m)L_{m+n}, \ \ & m> 0, m\neq -\nu,\\
(n+\nu)L_{n-\nu}+nbL_{n}, \ \ & m=-\nu,\\
0, & m\leqslant 0;
\end{cases}
$

{\rm ($\mathcal{NP}^{b,\nu}_3$):}  $\nu=-2, -3$ or $-4$,

\ \ \ \ \ \ $
L_m \circ_3 L_n=
\begin{cases}
(n-m)L_{m+n}, \ \ & m\geqslant 2,  \ m\neq -\nu\\
(n+\nu)L_{n-\nu}+nbL_{n}, \ \ & m=-\nu,\\
0, & m\leqslant 1;
\end{cases}
$

{\rm ($\mathcal{NP}^{b,\nu}_4$):}
$
L_m \circ_4 L_n=
\begin{cases}
(n-m)L_{m+n}, \ \ & m\geqslant 3,\\
(n-2)L_{n+2}+b(n-1)L_{n+1}, \ \ & m=2,\\
0, & m\leqslant 1;
\end{cases}
$

{\rm ($\mathcal{NP}^{b,\nu}_5$):}
$L_m \circ_5 L_n=
\begin{cases}
(n-m)L_{m+n}, \ \ & m\geqslant 4,\\
(n-2)L_{n+2}+nbL_{n}, \ \ & m= 2,\\
(n-3)L_{n+3}+2(n-1)bL_{n+1}, \ \ & m=3,\\
0, & m\leqslant 1;
\end{cases}
$

{\rm ($\mathcal{NP}^{b,\nu}_6$):}  $\nu=-2, -3$ or $-4$,

\ \ \ \ \ \ $
L_m \circ_6 L_n=
\begin{cases}
(n-m)L_{m+n}, & m\leqslant 1,\\
nbL_n, \ \ & m=-\nu,\\
0, \ \ & m\geqslant 2, m\neq -\nu;
\end{cases}
$

{\rm ( $\mathcal{NP}^{b,\nu}_7$):}
$
L_m \circ_7 L_n=
\begin{cases}
(n-m)L_{m+n}, & m\leqslant 1,\\
(n-1)bL_{n+1}, \ \ & m=2,\\
0, \ \ & m\geqslant 3;
\end{cases}
$

{\rm ($\mathcal{NP}^{b,\nu}_8$):}
 $L_m \circ_8 L_n=
\begin{cases}
(n-m)L_{m+n}, & m\leqslant 1,\\
nbL_{n}, \ \ & m=2,\\
2(n-1)bL_{n+1}, \ \ & m=3,\\
0, \ \ & m\geqslant 4;
\end{cases}
$

{\rm ($\mathcal{MP}^{b,\nu}_1$):}  $\nu=-1$ or $-2$,

 \ \ \ \ \ \ $
L_m \circ_9 L_n=
\begin{cases}
(n-m)L_{m+n}, \ \ & m\leqslant 0,\\
nbL_{n}, & m=-\nu,\\
0, & m>0,m\neq -\nu;
\end{cases}
$

{\rm ($\mathcal{MP}^{b,\nu}_2$):}  $\nu=1$ or $2$,

\ \ \ \ \ \  $
L_m \circ_{10} L_n=
\begin{cases}
(n-m)L_{m+n}, \ \ & m< 0, m\neq -\nu,\\
(n+\nu)L_{n-\nu}+nbL_{n}, \ \ & m=-\nu,\\
0, & m\geqslant 0;
\end{cases}
$

{\rm ($\mathcal{MP}^{b,\nu}_3$):}  $\nu=2, 3$ or $4$,

\ \ \ \ \ \ $
L_m \circ_{11} L_n=
\begin{cases}
(n-m)L_{m+n}, \ \ & m\leqslant -2,  \ m\neq -\nu\\
(n+\nu)L_{n-\nu}+nbL_{n}, \ \ & m=-\nu,\\
0, & m\geqslant -1;
\end{cases}
$

{\rm ($\mathcal{MP}^{b,\nu}_4$):}
$
L_m \circ_{12} L_n=
\begin{cases}
(n-m)L_{m+n}, \ \ & m\leqslant -3,\\
(n+2)L_{n-2}+b(n+1)L_{n-1}, \ \ & m=-2,\\
0, & m\geqslant -1;
\end{cases}
$

{\rm ($\mathcal{MP}^{b,\nu}_5$):}
$L_m \circ_{13} L_n=
\begin{cases}
(n-m)L_{m+n}, \ \ & m\leqslant -4,\\
(n+2)L_{n-2}+nbL_{n}, \ \ & m= -2,\\
(n+3)L_{n-3}+2(n+1)bL_{n-1}, \ \ & m=-3,\\
0, & m\geqslant -1;
\end{cases}
$

{\rm ($\mathcal{MP}^{b,\nu}_6$):}  $\nu=2, 3$ or $4$,

\ \ \ \ \ \ $
L_m \circ_{14} L_n=
\begin{cases}
(n-m)L_{m+n}, & m\geqslant -1,\\
nbL_n, \ \ & m=-\nu,\\
0, \ \ & m\leqslant -2, m\neq -\nu;
\end{cases}
$

{\rm ($\mathcal{MP}^{b,\nu}_7$):}
$
L_m \circ_{15} L_n=
\begin{cases}
(n-m)L_{m+n}, & m\geqslant -1,\\
(n+1)bL_{n-1}, \ \ & m=-2,\\
0, \ \ & m\leqslant -3;
\end{cases}
$

{\rm ($\mathcal{MP}^{b,\nu}_8$):} $
L_m \circ_{16} L_n=
\begin{cases}
(n-m)L_{m+n}, & m\geqslant -1,\\
nbL_{n}, \ \ & m=-2,\\
2(n+1)bL_{n-1}, \ \ & m=-3,\\
0, \ \ & m\leqslant -4
\end{cases}
$ \\
where $b$ is a non-zero number. Conversely, the above types are all shifting post-Lie algebra structures satisfying (\ref{mainequ11}) on the Witt algebra $W$. Furthermore, the post-Lie algebras $\mathcal{NP}^{b,\nu}_i$ are isomorphic to the post-Lie algebras
$\mathcal{MP}^{b,-\nu}_i$, $i=1, 2, \cdots, 8$  respectively, and other post-Lie algebras are not mutually isomorphic.
\end{theorem}

\begin{proof}
Suppose that $(W, [,], \circ)$ is a class of shifting post-Lie algebra structures satisfying (\ref{mainequ11}) on the Witt algebra $W$. By Propositions
\ref{jin}-\ref{pro32}, there are complex-valued function $f$ and $g$ on $\mathbb{Z}$ such that one of $16$ cases in Table 2 holds. Thus, by Lemma \ref{ffmm11} we know that the shifting post-Lie algebra structure must be one of the above $16$ types.  Conversely, every type of the $16$ cases means that there are complex-valued function $f$ and $g$ on $\mathbb{Z}$ such that (\ref{11211}) and (\ref{11222}) hold and, the Equations (\ref{11223})-(\ref{f(mn1)}) are easily verified. Thus, by Lemma \ref{ffmm11} we see that all they are the shifting post-Lie algebra structure satisfying (\ref{mainequ11}) on the Witt algebra $W$.   Finally, by Proposition \ref{iso2} we know that the post-Lie algebras $\mathcal{NP}^{b,\nu}_i$ are isomorphic to the post-Lie algebras
$\mathcal{MP}^{b,-\nu}_i$, $i=1, 2, \cdots, 8$  respectively. Next, in a similar way to the proof of Theorem \ref{mmgg}, one can see the other post-Lie algebras are not mutually isomorphic.
\end{proof}

\begin{proposition}\label{lie2}
Up to isomorphism, the post-Lie algebras in Theorem \ref{postnon1} give rise to the following Lie
algebras under the bracket \{, \} defined in (\ref{laoz}) of Proposition \ref{equ1}:

{\rm ($\mathcal{LNP}^{b,\nu}_1$):} $\nu=1$ or $\nu=2$,

\ \ \ \  $
\{L_m , L_n\}_1=
\begin{cases}
(n-m)L_{m+n}, \ \ & m,n\geqslant 0, \\
(m-n)L_{m+n}, \ \ & m,n<0, m,n\neq -\nu \\
nbL_n, & m=-\nu, n\geqslant 0,\\
nbL_n-(n+\nu)L_{n-\nu} &  m=-\nu, n<0, n\neq -\nu,\\
0,& \text{otherwise\ (unless some cases for}\ n=-\nu );
\end{cases}
$

{\rm ($\mathcal{LNP}^{b,\nu}_3$):} $\nu=-2,-3$ or $-4$,

\ \ \ \  $
\{L_m , L_n\}_3=
\begin{cases}
(n-m)L_{m+n}
\ \ & m,n \geqslant 2, m,n\neq -\nu, \\
(m-n)L_{m+n}, & m,n\leqslant 1, \\
nbL_{n}, & m=-\nu, n\leqslant 1,\\
nbL_n+(n+\nu)L_{n-\nu} &  m=-\nu, n\geqslant 2, n\neq -\nu,\\
0,& \text{otherwise\ (unless some cases for}\ n=-\nu );
\end{cases}
$

{\rm ($\mathcal{LNP}^{b,\nu}_4$):}
 $
\{L_m , L_n\}_4=
\begin{cases}
(n-m)L_{m+n}
\ \ & m,n \geqslant 3, \\
(m-n)L_{m+n}, & m,n\leqslant 1, \\
(n-1)bL_{n+1}, & m=2, n\leqslant 1,\\
(n-1)bL_{n+1}+(n-2)L_{n+2} &  m=2, n\geqslant 3,\\
0,& \text{otherwise\ (unless some cases for}\ n=2 );
\end{cases}
$

{\rm ($\mathcal{LNP}^{b,\nu}_5$):}
 $
\{L_m , L_n\}_5=
\begin{cases}
(n-m)L_{m+n}
\ \ & m,n \geqslant 4, \\
(m-n)L_{m+n}, & m,n\leqslant 1, \\
L_5+bL_3, & m=2, n=3,\\
nbL_n, & m=2, n\leqslant 1,\\
2(n-1)bL_{n+1}, & m=3, n\leqslant 1,\\
nbL_n+(n-2)L_{n+2} &  m=2, n\geqslant 4,\\
2(n-1)bL_{n+1}+(n-3)L_{n+3} &  m=3, n\geqslant 4,\\
0,& \text{otherwise\ (unless some cases for}\ n=2,3);
\end{cases}
$

where $b$ is a nonzero number.
\end{proposition}

\begin{proof}
The proof is similar to Proposition \ref{lie1}.
\end{proof}

\section{ Another understanding of the Post-Lie algebra structures}

We should see that there were two different definitions of post-Lie algebra structure, that is {\it post-Lie algebra structure on a Lie algebra} and {\it post-Lie algebra structure on  pairs of Lie algebras}. The former is studied as above. Now we recall the latter as follows.

\begin{definition}\label{defi[]{}} \cite{Burde2,Burde3}
Let $(V_1,[,])$ and $(V_2, \{,\})$ be a pair of Lie algebras on the same linear space $V=V_1=V_2$ over a field $k$.
A post-Lie algebra structure on the pair $(V_1, V_2)$ is a
$k$-bilinear product $x\circ y$ on $V$ satisfying the following identities:
\begin{eqnarray*}
&& <x,y> = \{x,y\}-[x,y], \label{post5}\\
&& \{x, y\} \circ z =  x \circ (y \circ z)-y \circ (x \circ z), \label{post6}\\
&& x\circ [y, z] = [x\circ y, z] + [y, x \circ z] \label{post7}
\end{eqnarray*}
for all $x, y, z \in V$, where $<x,y>=x\circ y-y\circ x$. We also say that $(V_1, V_2, \circ, [ , ], \{,\})$ is a post-Lie algebra.
\end{definition}

Inspired by the above definitions, here we would like to give another definition of a post-Lie algebra as follows. Below we will see that the three definitions of post-Lie algebra are equivalent.
\begin{definition}\label{defi{}}
A post-Lie algebra $(V, \circ, \{ , \})$ is a vector space $V$ over a field $k$ equipped with two
$k$-bilinear operations $x\circ y$ and $\{x, y\}$, such that $(V, \{, \})$ is a Lie algebra and
\begin{eqnarray}
&& \{x, y\} \circ z =  x \circ (y \circ z)-y \circ (x \circ z), \label{post3}\\
&& x\circ \{y, z\}- \{x\circ y, z\}- \{y, x \circ z\} = x\circ <y, z>-<x\circ y, z> - <y, x \circ z> \label{post4}
\end{eqnarray}
for all $x, y, z \in V$, where $<x,y>=x\circ y-y\circ x$. We also say that  $(V, \circ, \{ , \})$ is a post-Lie algebra structure on Lie algebra $(V, \{, \})$.
\end{definition}

\begin{proposition}\label{equ2}
A post-Lie algebra $(V, \circ, \{, \})$ defined by Definition \ref{defi{}} with
the following operation
$$
[x, y] \triangleq \{x, y\}-<x,y>.
$$
defines an another Lie algebra structure on $V$, where $<x,y>=x\circ y-y\circ x$.
\end{proposition}

\begin{proposition}\label{equequ}
Definitions \ref{defi[]},  \ref{defi[]{}} and \ref{defi{}} of the notion of post-Lie algebra are equivalent.
\end{proposition}

\begin{proof}
 If $(V, \circ, [, ])$ is a post-Lie algebra defined by Definition \ref{defi[]}, then by Proposition \ref{equ1} we know that under the Lie bracket $\{x, y\} =<x,y> + [x, y]$, $V$ admits a new Lie algebra structure. Obviously, $(V, \circ, \{, \})$ satisfies (\ref{post3}) and (\ref{post4}). Conversely, when $(V, \circ, \{, \})$  is a post-Lie algebra defined by Definition \ref{defi{}}, then by Proposition \ref{equ2} we know that under the Lie bracket $[x, y] =\{x,y\} - <x, y>$, $V$ also admits a new Lie algebra structure. We are able to verily that $(V, \circ, [, ])$ satisfies (\ref{post1}) and (\ref{post2}). This tell us that whether $(V, \circ, [, ])$ or $(V, \circ, \{, \})$  all are connotations of two Lie algebras structures, which satisfy the conditions of Definitions \ref{defi[]{}}. On the other hand, a post-Lie algebra structure on the pair $(V_1, V_2)$ defined by Definitions \ref{defi[]{}} imply $(V_1, \circ, [, ])$ satisfying (\ref{post1}) and (\ref{post2}) or $(V_2, \vartriangleright, \{, \})$ satisfying (\ref{post3}) and (\ref{post4}).
\end{proof}

\begin{remark} \label{remark2}
 Recall that the left multiplications of the algebra $A = (V,\circ)$ are denoted by $\mathcal{L}(x)$, i.e., we have $\mathcal{L}(x)(y) = x\circ y $ for all $x,y\in V$.  Clearly, by (\ref{post3}) we see that the map $\mathcal{L}: V \rightarrow End(V )$ given by $x\mapsto \mathcal{L}(x)$ is a linear representation of the Lie algebra
  $(V,\{, \})$ when $(V,\circ, \{,\})$ is the post-Lie algebra defined by Definition \ref{defi{}}.
\end{remark}

It can be seen that the study of post-Lie algebra structures on pairs of Lie algebras given by Definition \ref{defi[]{}} is divided into two directions: either when $(V_1, [,])$ is a given Lie algebra, to determine the product $\circ$; or when $(V_2, \{,\})$ is a given Lie algebra, to determine the product $\circ$. By Proposition \ref{equequ}, the first direction is characterizing of the post-Lie algebra structures  $(V_1, [,], \circ)$ on the Lie algebra $(V_1, [,])$ given by Definition \ref{defi[]}, and another direction is characterizing of the post-Lie algebra structure  $(V_2, \{,\}, \circ)$ on the Lie algebra $(V_2, \{,\})$ given by Definition \ref{defi{}}. For the first case in which $V_1=W$ is the Witt algebra, the graded or some shifting post-Lie algebra structures are studied in Sections \ref{grad} and \ref{nograd}. The problem of another case in which $V_2=W$ is the Witt algebra, should be interesting. We are not going to discuss this problem here. But, inspired by \cite{KCB,tangb2012}, we may give two non-trivial examples as follows.

\begin{example}
The following cases give two class of graded post-Lie algebra
structures on $W$ satisfying (\ref{mainequ1}) given by Definition \ref{defi{}}.
\begin{equation*} %\label{20101009-1}
\phi(m,n)=-\frac{(\alpha+n+\alpha\epsilon m)(1+\epsilon
n)}{1+\epsilon(m+n)},\forall m,n\in \mathbb Z,\label{kcb-f1}
\end{equation*}
 where $\alpha, \epsilon\in
\mathbb{C}$ satisfy $\epsilon=0$ or $\epsilon^{-1}\not\in
\mathbb{Z}$, or
\begin{equation*} %\label{20101009-2}
\phi(m,n)= \left\{\begin{array}{ll} -n-t, & {\mbox{if}}\  m+n+t\neq 0,\\
 \frac{(n+t)(n+t-\beta)}{\beta-t},&{\mbox{if}}\ m+n+t=
 0,\end{array}\right.\forall m,n\in \mathbb Z,\label{kcb-f2}
\end{equation*}
where $\beta \in \mathbb{C}$ and $t\in \mathbb{Z}$ satisfy
$\beta\neq t$.
\end{example}

\begin{example}
Let $\phi(m,n)=-(n+\alpha)$ and $\varrho(m,n)=\mu$ in (\ref{mainequ11}), where $\alpha, \mu\in \mathbb{C}$. This is a shifting post-Lie algebra structure on the Witt algebra given by Definition \ref{defi{}} as follows.
$$
L_m\circ L_n=-(n+\alpha)L_{m+n}+\mu L_{m+n+\nu}.
$$
\end{example}

By Propositions \ref{lie1} and \ref{lie2}, we find $7$ classes of Lie algebras up to isomorphism. They are $\mathcal{LP}_{1}$, $\mathcal{LP}^a_{3}$, $\mathcal{LP}_{5}$, $\mathcal{LNP}^{b,\nu}_1$, $\mathcal{LNP}^{b,\nu}_3$, $\mathcal{LNP}^{b,\nu}_4$  and $\mathcal{LNP}^{b,\nu}_5$. Note that
 Remark \ref{remark2} tells us that the post-Lie $(L, \circ, \{,\})$ admits a module structure of Lie algebra $(L, \{,\})$. Thus, by Theorems \ref{mmgg} and \ref{postnon1}, we have the following results on module structures of some Lie algebras.

\begin{proposition}
Let $\mathcal{V}={\rm Span}_{\mathbb{C}}\{v_i|i\in \mathbb{Z}\}$ be a linear vector space over the complex number field.
Then $\mathcal{V}$  becomes a module over some Lie algebras under the following acts.

{\rm (1)} For the Lie algebra $\mathcal{LP}_{1}$: $L_m . v_n=0, \ \forall m,n\in \mathbb{Z}$.

{\rm (2)} For the Lie algebra $\mathcal{LP}^a_{3}$:

\ \ \ \ $L_m . v_n=\begin{cases}
(n-m)v_{m+n}, &  m> 0,\\
nav_n, & m=0, \\
0, & m<0;
\end{cases}
$

{\rm (3)} For the Lie algebra $\mathcal{LP}_5$:

\ \ \ \ $L_m . v_n=\begin{cases}
(n-m)v_{m+n}, & m\geqslant 2,\\
0, & m\leqslant 1;
\end{cases}
$

 {\rm (4)} For the Lie algebra $\mathcal{LNP}^{b,\nu}_1$: $\nu=1$ or $2$,

 \ \ \ \  $
L_m . v_n=
\begin{cases}
(n-m)v_{m+n}, \ \ & m\geqslant 0,\\
nbv_{n}, & m=-\nu,
\\
0, & m<0,m\neq -\nu;
\end{cases}
$

{\rm (5)} For the Lie algebra $\mathcal{LNP}^{b,\nu}_3$: $\nu=-2, -3$ or $-4$,

\ \ \ \  $
L_m . v_n=
\begin{cases}
(n-m)v_{m+n}, \ \ & m\geqslant 2,  \ m\neq -\nu\\
(n+\nu)L_{n-\nu}+nbL_{n}, \ \ & m=-\nu,\\
0, & m\leqslant 1;
\end{cases}
$

{\rm (6)} For the Lie algebra $\mathcal{NP}^{b,\nu}_4$:

\ \ \ \  $
L_m. v_n=
\begin{cases}
(n-m)v_{m+n}, \ \ & m\geqslant 3,\\
(n-2)v_{n+2}+b(n-1)v_{n+1}, \ \ & m=2,\\
0, & m\leqslant 1;
\end{cases}
$

{\rm (7)} For the Lie algebra $\mathcal{NP}^{b,\nu}_5$:

\ \ \ \ $L_m . v_n=
\begin{cases}
(n-m)v_{m+n}, \ \ & m\geqslant 4,\\
(n-2)v_{n+2}+nbv_{n}, \ \ & m= 2,\\
(n-3)v_{n+3}+2(n-1)bv_{n+1}, \ \ & m=3,\\
0, & m\leqslant 1,
\end{cases}
$\\
where $a,b\in \mathbb{C}$ with $b\neq 0$.
\end{proposition}

\section{Application to Rota-Baxter operators}

Now let us recall the definition of Rota-Baxter operator.

\begin{definition}  Let $L$ be a complex Lie algebra. A Rota-Baxter operator of weight $\lambda \in \mathbb{C}$ is a linear map $R: L\rightarrow L$ satisfying
\begin{equation}\label{def-rbalgebra}
[R(x),R(y)]=R([R(x),y]+[x,R(y)])+\lambda R([x,y]), \forall x,y\in L.
\end{equation}
\end{definition}
  Note that if $R$ is a Rota-Baxter operator of weight $\lambda\neq 0$, then $\lambda^{-1}R$ is a Rota-Baxter operator of weight $1$. Therefore, one only needs to consider Rota-Baxter operators of weight $0$ and $1$.

\begin{lemma}\label{rota} \cite{BGN}
Let $L$ be a complex Lie algebra  and $R : L\rightarrow L$ a Rota-Baxter operator of weight $1$. Define a new operation
$x\circ y$ = $[R(x), y]$ on $L$. Then $(L, [, ], \circ)$ is a post-Lie algebra given by Definition \ref{defi[]}.
\end{lemma}

In this section, by use of Lemma \ref{rota}, Theorems \ref{mmgg} and \ref{postnon1}, we mainly consider the homogeneous Rota-Baxter operator of weight $1$ on the Witt algebra $W$ such that
$R: W\rightarrow W$ given by
\begin{equation}\label{r1}
R(L_m)=f(m)L_m, \ \forall m\in \mathbb{Z}
\end{equation}
and the non-homogeneous Rota-Baxter operator of weight $1$ on the Witt algebra $W$ such that
\begin{equation}\label{r2}
R(L_m)=f(m)L_m+g(m)L_{m+\nu}, \ \forall m\in \mathbb{Z}
\end{equation}
where $\nu$ is a nonzero integer number and $f,g$ are complex-valued functions on $\mathbb{Z}$ with $g\neq 0$.

Up till now,  the authors \cite{Galb} have presented the homogeneous Rota-Baxter operators on the Witt algebras. Inspired by this, we will prove the following.

\begin{theorem}\label{rb1}
A homogeneous Rota-Baxter operator of weight $1$ satisfying (\ref{r1}) on the Witt algebra $W$ must be one of the
following types

{\rm ($\mathcal{R}_1$):}
$
R(L_m)=0, \forall m\in \mathbb{Z};
$

{\rm ($\mathcal{R}_2$):}
$
R(L_m)=-L_{m}, \forall m\in \mathbb{Z};
$

{\rm ($\mathcal{R}^a_3$):}
$
R(L_m)=
\begin{cases}
-L_{m}, \ \ & m>0,  \\
aL_{0}, \ \ & m=0,\\
0, & m< 1,
\end{cases}
$

{\rm ($\mathcal{R}^a_4$):}
$
R(L_m)=
\begin{cases}
-L_{m}, \ \ & m< 0,\\
aL_{0}, \ \ & m=0,\\
0, & m>0,
\end{cases}
$

{\rm ($\mathcal{R}_5$):}
$R(L_m)=
\begin{cases}
-L_{m}, \ \ & m\geqslant 2,\\
0, \ \ & m\leqslant 1;
\end{cases}
$

{\rm ($\mathcal{R}_6$):}  $
R(L_m)=
\begin{cases}
-L_{m}, & m\leqslant 1,\\
0, \ \ & m\geqslant 2;
\end{cases}
$

{\rm ($\mathcal{R}_7$):}
$
R(L_m)=
\begin{cases}
-L_{m}, & m\geqslant -1,\\
0, \ \ & m\leqslant -2;
\end{cases}
$

{\rm ($\mathcal{R}_8$):}  $
R(L_m)=
\begin{cases}
-L_{m}, & m\leqslant -2,\\
0, \ \ & m\geqslant -1,
\end{cases}
$\\
where $a\in \mathbb{C}$.
\end{theorem}

\begin{proof}
Due to Lemma \ref{rota}, if we define a new operation $x\circ y = [R(x), y]$ on $W$, then $(W, [,], \circ)$ is  a post-Lie algebra. By (\ref{r1}) we have
$$
L_m\vartriangleright L_n=[R(L_m), L_n]=(m-n)f(m)L_{m+n}, \ \forall m,n\in \mathbb{Z}.
$$
This means that $(W, [,], \vartriangleright)$ is a graded post-Lie algebra structure satisfying (\ref{mainequ1}) on $W$ with $\phi(m,n)=(m-n)f(m)$.
By Theorem \ref{mmgg}, we see that $f$ must be one of the eight cases listed in Table 1, which can be get the eight forms of $R$ one by one.
On the other hand, it is easy to verify that every form of $R$ listed in the above is a Rota-Baxter operator of weight $1$ satisfying (\ref{r1}).
The proof is completed.
\end{proof}

\begin{theorem}\label{rb2}
 A non-homogeneous Rota-Baxter operator of weight $1$ satisfying (\ref{r2}) on the Witt algebra $W$ must be one of the
following types

 {\rm ($\mathcal{NR}^b_1$):}  $\nu=1$ or $2$,

 \ \ \ \ \ \ $
R(L_m)=
\begin{cases}
-L_{m}, \ \ & m\geqslant 0,\\
-bL_{0}, & m=-\nu,
\\
0, & m<0,m\neq -\nu;
\end{cases}
$

{\rm ( $\mathcal{NR}^b_2$):} $\nu=-1$ or $-2$,

\ \ \ \ \ \  $
R(L_m)=
\begin{cases}
-L_{m}, \ \ & m> 0, m\neq -\nu,\\
-L_{-\nu}-bL_{0}, \ \ & m=-\nu,\\
0, & m\leqslant 0;
\end{cases}
$

{\rm ($\mathcal{NR}^b_3$):} $\nu=-2, -3$ or $-4$,

\ \ \ \ \ \ $
R(L_m)=
\begin{cases}
-L_{m}, \ \ & m\geqslant 2,  \ m\neq -\nu\\
-L_{-\nu}-bL_{0}, \ \ & m=-\nu,\\
0, & m\leqslant 1;
\end{cases}
$

{\rm ($\mathcal{NR}^b_4$):}
$
R(L_m)=
\begin{cases}
-L_{m}, \ \ & m\geqslant 3,\\
-L_{2}-bL_{1}, \ \ & m=2,\\
0, & m\leqslant 1;
\end{cases}
$

{\rm ($\mathcal{NR}^b_5$):}  $
R(L_m)=
\begin{cases}
-L_{m}, \ \ & m\geqslant 4,\\
-L_{2}-bL_{0}, \ \ & m= 2,\\
-L_{3}-2bL_{1}, \ \ & m=3,\\
0, & m\leqslant 1;
\end{cases}
$

{\rm ($\mathcal{NR}^b_6$):}  $\nu=-2, -3$ or $-4$,

\ \ \ \ \ \ $
R(L_m)=
\begin{cases}
-L_{m}, & m\leqslant 1,\\
-bL_0, \ \ & m=-\nu,\\
0, \ \ & m\geqslant 2, m\neq -\nu;
\end{cases}
$

{\rm ($\mathcal{NR}^b_7$):}
$
R(L_m)=
\begin{cases}
-L_{m}, & m\leqslant 1,\\
-bL_{1}, \ \ & m=2,\\
0, \ \ & m\geqslant 3;
\end{cases}
$

{\rm ( $\mathcal{NR}^b_8$):} $
R(L_m)=
\begin{cases}
-L_{m}, & m\leqslant 1,\\
-bL_{0}, \ \ & m=2,\\
-2bL_{1}, \ \ & m=3,\\
0, \ \ & m\geqslant 4;
\end{cases}
$

{\rm ($\mathcal{MR}^b_1$):}  $\nu=-1$ or $-2$,

 \ \ \ \ \ \ $
R(L_m)=
\begin{cases}
-L_{m}, \ \ & m\leqslant 0,\\
-bL_{0}, & m=-\nu,\\
0, & m>0,m\neq -\nu;
\end{cases}
$

{\rm ($\mathcal{MR}^b_2$):}  $\nu=1$ or $2$,

\ \ \ \ \ \  $
R(L_m)=
\begin{cases}
-L_{m}, \ \ & m< 0, m\neq -\nu,\\
-L_{-\nu}-bL_{0}, \ \ & m=-\nu,\\
0, & m\geqslant 0;
\end{cases}
$

{\rm ($\mathcal{MR}^b_3$):}  $\nu=2, 3$ or $4$,

\ \ \ \ \ \ $
R(L_m)=
\begin{cases}
-L_{m}, \ \ & m\leqslant -2,  \ m\neq -\nu\\
-L_{-\nu}-bL_{0}, \ \ & m=-\nu,\\
0, & m\geqslant -1;
\end{cases}
$

{\rm ($\mathcal{MR}^b_4$):}
$
R(L_m)=
\begin{cases}
-L_{m}, \ \ & m\leqslant -3,\\
-L_{-2}-bL_{-1}, \ \ & m=-2,\\
0, & m\geqslant -1;
\end{cases}
$

{\rm ($\mathcal{MR}^b_5$):}  $
R(L_m)=
\begin{cases}
-L_{m}, \ \ & m\leqslant -4,\\
-L_{-2}-bL_{0}, \ \ & m= -2,\\
-L_{-3}-2bL_{-1}, \ \ & m=-3,\\
0, & m\geqslant -1;
\end{cases}
$

{\rm ($\mathcal{MR}^b_6$):}  $\nu=2, 3$ or $4$,

\ \ \ \ \ \ $
R(L_m)=
\begin{cases}
-L_{m}, & m\geqslant -1,\\
-bL_0, \ \ & m=-\nu,\\
0, \ \ & m\leqslant -2, m\neq -\nu;
\end{cases}
$

{\rm ($\mathcal{MR}^b_7$):}
$
R(L_m)=
\begin{cases}
-L_{m}, & m\geqslant -1,\\
-bL_{-1}, \ \ & m=-2,\\
0, \ \ & m\leqslant -3;
\end{cases}
$

{\rm ($\mathcal{MR}^b_8$):}  $
R(L_m)=
\begin{cases}
-L_{m}, & m\geqslant -1,\\
-bL_{0}, \ \ & m=-2,\\
-2bL_{-1}, \ \ & m=-3,\\
0, \ \ & m\leqslant -4,
\end{cases}
$\\
where $b$ is a non-zero number. Conversely,the above operators are all the non-homogeneous Rota-Baxter operators of weight $1$ satisfying (\ref{r2}) on the Witt algebra $W$.
\end{theorem}

\begin{proof}
Due to Lemma \ref{rota}, if we define a new operation $x\circ y = [R(x), y]$ on $W$, then $(W, [,], \circ)$ is  a post-Lie algebra. By (\ref{r2}) we have
$$
L_m\circ L_n=[R(L_m), L_n]=(m-n)f(m)L_{m+n}+(m-n+\nu)g(m)L_{m+n+\nu}, \ \forall m,n\in \mathbb{Z}.
$$
This means that $(W, [,], \circ)$ is a shifting post-Lie algebra structure satisfying (\ref{mainequ2}) on $W$ with $\phi(m,n)=(m-n)f(m)$
and $\varrho(m,n)=(m-n+\nu)g(m)$.
By Theorem \ref{postnon1}, we see that $f,g$ and $\nu$ must be one of the $16$ cases listed in Table 2, which can get the $16$ forms of $R$ one by one. On the other hand, it is easy to verify that every form of $R$ listed in the above is a Rota-Baxter operator of weight $1$ satisfying (\ref{r2}).
The proof is completed.
\end{proof}

\begin{remark}
The Rota-Baxter operators given by Theorem \ref{rb1} just are the all homogeneous Rota-Baxter operators of weight $1$ described in \cite{Galb}.  But the Rota-Baxter operators given by Theorem \ref{rb2} are new and non-homogeneous.
\end{remark}

\begin{remark}
The Rota-Baxter operators on the Witt algebra $W$ can be given a class of solutions of the classical
Yang-Baxter equation (CYBE) on $W\ltimes_{{\rm ad}^\ast} W^\ast$. The details can be found in \cite{Galb}, which
discuss the homogeneous case. Along the same lines, we can also consider the non-homogeneous case by use of Theorem \ref{rb2}. It is not discussed here.
\end{remark}

\section*{Acknowledgments}
This work is supported in part by NSFC (Grant No. 11771069), NSF of Heilongjiang Province (Grant No. A2015007), the Fund of Heilongjiang Education Committee (Grant No. 12531483 and No. HDJCCX-2016211).

\end{document}